\documentclass[11pt,a4paper]{article}

\usepackage{indentfirst,csquotes}
\usepackage{authblk}

\usepackage[bookmarks=false]{hyperref}
   \hypersetup{colorlinks=true,
 linkcolor=blue!80,
 citecolor=orange!70!black,
 urlcolor=blue!80}

\usepackage{a4wide}

\usepackage{amsthm,mathtools}
\usepackage{amsmath,amssymb}
\usepackage{amsfonts}
\usepackage[thmtools-compat]{keytheorems}
\usepackage{zref-clever}
\zcsetup{
       cap = true,
    abbrev = false,
    nameinlink,
}

\usepackage{xcolor,paralist,titlesec,fancyhdr,etoolbox}
\usepackage[figuresright]{rotating}

\newcommand{\rg}{$\{2\}$-Roman graph}
\newcommand{\kr}{$\{k\}$-Roman\ }
\newcommand{\gr}{\gamma_{\{R2\}}}
\newcommand*{\abs}[1]{\ensuremath{{\lvert {#1} \rvert}}}

\newtheorem{theorem}{Theorem}[section]
\newtheorem{lemma}[theorem]{Lemma}
\newtheorem{observation}[theorem]{Observation}
\newtheorem{corollary}[theorem]{Corollary}
\newtheorem{definition}{Definition}[section]

\newtheorem{claim}[theorem]{Claim}

\newenvironment{subproof}[1][Proof.]{%
    \begin{proof}[{#1}]%
        }{%
    \end{proof}}
\usepackage{comment}

\newcommand{\sun}[1]{S_{#1}}
\newcommand{\cosun}[1]{\overline{S}_{#1}}

\begin{document}
\title{On $\{k\}$-Roman graphs: complexity of recognition and the case of split graphs\thanks{An extended abstract of this work was accepted for the proceedings of the XIII Latin American Algorithms, Graphs, and Optimization Symposium (LAGOS 2025).}\thanks{Due to an update of this work, some of the former results are only available in previous ArXiv versions}}


\author[1,2]{Kenny Bešter Štorgel}
\author[2,3]{Nina Chiarelli}
\author[4,5]{Lara Fernández\thanks{lara@fceia.unr.edu.ar}$^,$}
\author[2]{J.~Pascal Gollin\thanks{pascal.gollin@famnit.upr.si}$^,$}
\author[2]{Claire Hilaire}
\author[4,5]{Valeria Leoni}
\author[2,3]{Martin Milanič \thanks{martin.milanic@upr.si}$^,$}
\affil[1]{Faculty of Information Studies in Novo Mesto, Slovenia}
\affil[2]{FAMNIT, University of Primorska, Koper, Slovenia}
\affil[3]{IAM, University of Primorska, Koper, Slovenia}
\affil[4]{Universidad Nacional de Rosario, Argentina}
\affil[5]{CONICET, Argentina}

\maketitle
 
\begin{abstract}
For a positive integer $k$, a $\{k\}$-Roman dominating function of a graph $G = (V,E)$ is a function $f\colon V \rightarrow \{0,1,\ldots,k\}$ satisfying $\sum_{u\in N(v)} f(u) \geq k$ for each vertex $v\in V$ with $f (v) = 0$. 
Every graph $G$ satisfies $\gamma_{\{Rk\}}(G) \leq k\gamma(G)$, where $\gamma(G)$ is the domination number of $G$ and $\gamma_{\{Rk\}}(G)$ denotes the $\{k\}$-Roman domination number of $G$, that is, the minimum value of $\sum_{u\in V(G)} f(u)$ over all $\{k\}$-Roman dominating functions of $G$. 
 In this work we study graphs for which the equality is reached, called \emph{$\{k\}$-Roman graphs}.
This extends the concept of $\{k\}$-Roman trees studied by Wang et al.~in 2021 to general graphs.
We prove that for every $k\geq 2$, the problem of recognizing \hbox{$\{k\}$-Roman} graphs is \textsf{NP}-hard, even for split graphs. 
For ${k\geq 3}$, we give an alternative proof by generalizing several known results on domination in middle graphs to the hypergraph setting.
Finally, we characterize the \kr  property within two specific subclasses of split graphs: suns and their complements. 
\end{abstract}

%\begin{keyword} 
%$\{k\}$-Roman domination \sep $\{k\}$-Roman graph \sep \textsf{NP}-hardness \sep split graph \sep middle graph \sep hypergraph

%\MSC 05C69 
% Vertex subsets with special properties (dominating sets, independent sets, cliques, etc.)
%\sep 
%68Q25 % Analysis of algorithms and problem complexity
%\sep
%05C75 
% Structural characterization of families of graphs
%\sep
%05C65 
% Hypergraphs
%\end{keyword}

\section{Introduction}\label{sec:intro}

\begin{sloppypar}

Domination in graphs is a rich and fast developing area of graph theory, with motivations stemming from both practical and theoretical considerations (see, e.g., \cite{Haynes1998}).

For a graph $G=(V,E)$ and a vertex $v\in V$, we denote by $N_G(v)$ the \emph{neighborhood} of $v$ in $G$, that is, the set of all the vertices adjacent to $v$ in $G$, and by $N_G[v]$ its \emph{closed neighborhood}, i.e., the set $N_G(v)\cup\{v\}$.
We may omit the subscript $G$ when the graph is clear from the context. 
A \emph{dominating set} in a graph $G = (V,E)$ is a set $D\subseteq V$ such that $D\cap N[v]\neq \emptyset$ for all $v\in V$.
The \emph{domination number} of $G$, denoted by $\gamma(G)$, is the minimum cardinality of a dominating set in $G$. 

Over the years, many variants of the classical concept of domination were proposed from the neighborhood-sum point of view~\cite{Haynes1998}. 
One of those variants is $\{k\}$-domination, where $k$ is a positive integer.
Given a graph $G=(V,E)$, a function $f\colon V \rightarrow \{0, 1,\ldots, k\}$ is called a \emph{$\{k\}$-dominating function} of $G$ if every $v \in  V$ satisfies $f(N [v]) \geq k$~\cite{Domke1991}\footnote{For a function $f$ defined on $V$ and a set $S\subseteq V$, we denote the \emph{weight} of $S$ under $f$ by $f(S) = \sum_{v\in S} f(v)$. 
The \emph{weight} of $f$ is defined as $f(V)$.}.
In 2016, Chellali et al.~\cite{Chellali2016} introduced $\{2\}$-Roman domination\footnote{The concept of $\{k\}$-Roman domination has been studied in the literature under the name ``Roman $\{k\}$-domination''; however, since $\{k\}$-domination is a concept that is indeed changed and not just made more specific by the adjective Roman, we find it more sensible to have ``$\{k\}$-Roman'' as the adjective.} ---also called Italian domination in later works (see, e.g.,~\cite{Henning2017})--- as a variant of $\{k\}$-domination for $k=2$.
A function $f\colon V \rightarrow \{0, 1, 2\}$ is called a \emph{$\{2\}$-Roman dominating function} of $G$ if  $f(N (v)) \geq 2$ for each  $v \in  V$ with $f(v)=0$.
The minimum weight of a \hbox{$\{2\}$-Roman} dominating function of $G$ is denoted by $\gamma_{\{R2\}}(G)$. 
The $\{2\}$-Roman domination was also studied in 2007 by Bre\v{s}ar and Kraner \v{S}umenjak~\cite{Bresar2007} under the name \emph{weak $\{2\}$-domination}.
Several \textsf{NP}-completeness results as well as polynomial cases of the $\{2\}$-Roman domination problem are presented in~\cite{Chellali2016,Chakradhar2020,Poureidi2020}, and the  most recent ones in~\cite{Fernandez2023}.
\end{sloppypar}

Every graph $G$ satisfies $\gamma(G) \leq \gamma_{\{R2\}}(G) \leq 2 \gamma(G)$ (see, e.g.,~\cite{Chellali2016}).  
A graph $G$ is said to be a \emph{$\{2\}$-Roman graph} (or an \emph{Italian graph}) if $\gamma_{\{R2\}}(G) = 2\gamma(G)$~\cite{Klostermeyer}. 
A characterization of $\{2\}$-Roman trees in terms of four graph operations is  given in~\cite{Henning2017}. 
The $\{2\}$-Roman property of graphs is characterized in~\cite{Ferrari2025} by the existence of a minimum $\{2\}$-Roman dominating function of $G$ that assumes only $0, 2$-values. 

\begin{sloppypar}
Inspired by the work in~\cite{Chellali2016}, the ``Roman'' variant of $\{k\}$-domination for any positive integer $k$ is introduced and studied in 2021 by Wang et al.~\cite{Wang2021} (see also the earlier work of Chang et al.~\cite{Chang2013} where the problem was studied under the name ``weak $\{k\}$-domination'').
A \emph{$\{k\}$-Roman dominating function} of a graph $G = (V,E)$ is a function $f\colon V \rightarrow \{0, 1,\ldots,k\}$ satisfying $f (N(v)) \geq k$ for each vertex $v\in V$ with
$f (v) = 0$. 
The minimum weight of a $\{k\}$-Roman dominating function of $G$ is called the \hbox{\emph{$\{k\}$-Roman domination number}} of $G$ and denoted by $\gamma_{\{Rk\}}(G)$.
When a function $f$ is a $\{k\}$-Roman dominating function of a graph $G$ with weight $\gamma_{\{Rk\}}(G)$, we simply say that $f$ is a \hbox{\emph{$\gamma_{\{Rk\}}$-function}} of $G$. 
The decision problem associated with $\{k\}$-Roman domination is \textsf{NP}-complete, even for bipartite planar graphs, chordal bipartite graphs, and undirected path graphs~\cite{Wang2021}.
\end{sloppypar}

Among the several bounds proved in~\cite{Wang2021} for the $\{k\}$-Roman domination number  is the inequality $\gamma_{\{Rk\}}(G) \leq k\gamma(G)$, valid for every positive integer $k$ and every graph $G$ (see also~\cite{Chang2013}). 
For an integer $k\ge 2$, we say that a graph $G$ is a \emph{$\{k\}$-Roman graph} if $\gamma_{\{Rk\}}(G) = k\gamma(G)$.\footnote{The inequality $k\ge 2$ in the definition is justified by the fact that every graph $G$ satisfies $\gamma_{\{R1\}}(G) = \gamma(G)$.}
In~\cite{Wang2021}, $\{k\}$-Roman trees are characterized in terms of five operations.
Except for this result, however, nothing is known about $\{k\}$-Roman graphs, and it is the aim of this work to partially fill this gap.

Most of our results deal with the class of split graphs.
A graph $G$ is said to be \emph{split} if it admits a \emph{split partition}, that is, a partition $(K,I)$ of the vertex set such that $K$ and $I$ are respectively a clique and an independent set in $G$. 
Note that a split graph may admit more than one split partition (see, e.g.,~\cite[Chapter 6]{MR2063679}).

As our main result, we show in Section \ref{sec:recognition3SAT} that for every $k\geq 2$, the problem of recognizing $\{k\}$-Roman graphs is \textsf{NP}-hard, even when restricted to the class of split graphs.
The case $k = 2$ is similar to a result of Wang et al.~\cite{DBLP:conf/aaim/WangYLX24}, who showed \textsf{NP}-hardness of the problem of recognizing graphs achieving equality in the inequality $\gamma_{\{R2\}}(G)\le \gamma_2(G)$, where $\gamma_2(G)$ denotes the \emph{$2$-domination number} of a graph $G$, that is, the smallest cardinality of a set $S\subseteq V(G)$ such that every vertex in $V(G)\setminus S$ has at least two neighbors in~$S$.
Further \textsf{NP}-hardness results of similar flavor were obtained by Alvarado, Dantas, and Rautenbach~\cite{MR3336112,MR3522148,MR3684117,MR4222713}, Boruzanl\i\ Ekinci and Bujt\'as~\cite{MR4686207}, and Milanič~\cite{MR2928447}. 

In \zcref{sec:recognition} we generalize several known results on domination and $\{2\}$-Roman domination in middle graphs to the hypergraph setting, by connecting the existence of perfect matchings in $k$-uniform hypergraphs with the property of being a $\{k\}$-Roman graph.
As a by-product, this gives an alternative proof of the \textsf{NP}-hardness of recognizing $\{k\}$-Roman split graphs for $k\geq 3$.

In the last section, we use some of the results from  \zcref{sec:recognition} to characterize the \kr   property within two specific subclasses of split graphs: suns and their complements.
Furthermore, some of the proofs in that section rely on the following general observation of independent interest relating $\{k\}$-Roman graphs for different values of $k$: classes of $\{k\}$-Roman graphs, with $k\ge 2$, form a nonincreasing sequence of classes, in the sense that every $\{k+1\}$-Roman graph is also $\{k\}$-Roman.

Some of the results presented here appeared in an extended abstract published in the proceedings of the XIII Latin American Algorithms, Graphs, and Optimization Symposium (LAGOS 2025)~\cite{STORGEL2025325}.

\section{Recognizing $\{k\}$-Roman split graphs}
\label{sec:recognition3SAT}

In this section, we show that for all $k\ge 2$, the problem of recognizing \hbox{$\{k\}$-Roman} graphs is \textsf{NP}-hard, even when restricted to the class of split graphs.
We adapt the approach of Wang et al.~\cite{DBLP:conf/aaim/WangYLX24}, who considered sharpness of the inequality $\gamma_{\{R2\}}(G) \leq \gamma_2(G)$, to the inequality ${\gamma_{\{Rk\}}(G) \leq k\gamma(G)}$, and to the class of split graphs.

We use a reduction from the \textsc{3-Satisfiability} ($3$-SAT) problem, which we now recall.
The input to the problem is a pair $(X,C)$ where $X = \{x_1,\ldots, x_n\}$ is a set of $n$ Boolean variables and $C = \{c_1,\ldots, c_m\}$ is a set of $m$ clauses over $X$, each formed as a disjunction of exactly three literals over $X$, say $c_i = \ell^i_{1}\vee \ell^i_{2}\vee \ell^i_{3}$, where $\ell^i_{j}\in \{x_1,\ldots, x_n,\neg x_1,\ldots, \neg x_n\}$ for all $i \in \{1,\ldots, m\}$ and $j\in \{1,2,3\}$.
The task is to determine whether the conjunction of all the clauses is satisfiable, that is, whether it admits a satisfying assignment (an assignment of true/false values to the variables in $X$ such that all clauses evaluate to true).
The $3$-SAT problem is \textsf{NP}-complete (see~\cite{MR378476}). 

Our reduction relies on the notion of a \emph{complete split graph}, that is, a split graph that admits a split partition $(K,I)$ such that each vertex of $K$ is adjacent to every vertex of $I$.

\begin{theorem}\label{thm:k-Roman-NP-hard-via-3SAT}
For every fixed $k\ge 2$, the problem of recognizing $\{k\}$-Roman graphs is \textsf{NP}-hard, even when restricted to the class of split graphs.
\end{theorem}

\begin{proof}
Let~$k\ge 2$.
We prove \textsf{NP}-hardness by a reduction from $3$-SAT.
Let an instance of $3$-SAT with variables~$x_1, \dots, x_n$ and clauses~$c_1, \dots, c_m$ be given, and let ${\phi \coloneqq \bigwedge_{i = 1}^mc_i}$ be the corresponding Boolean formula.
For each~${i \in [n]}$\footnote{$[n]$ indicates the set of integers between $1$ and $n$, i.e., the set $\{1, \ldots, n\}$.}, we define a complete split graph~$H_i$ with vertex set~$\{v_i^\top, v_i^\bot, v_i^1, \dots, v_i^k\}$, where~$\{v_i^\top, v_i^\bot\}$ is a clique and~$\{v_i^1, \dots, v_i^k\}$ is an independent set forming a split partition of $H_i$. 
Now we construct a graph~$G$, as follows.
We start with the disjoint union of all these graphs~$H_i$ together with, for each~${j \in [m]}$, a vertex~$u_j$ (representing the clause~$c_j$) with an edge from~$u_j$ to~$v_i^\top$ if~$x_i$ appears in the clause~$c_j$ and an edge from~$u_j$ to~$v_i^\bot$ if~$\neg x_i$ appears in the clause~$c_j$. 
Lastly, we add edges so that the set~$\bigcup_{i = 1}^n\{v_i^\top, v_i^\bot\}$ becomes a clique. 
In particular, since the remaining vertices form an independent set, $G$ is a split graph. 

\begin{claim}\label{cl:3satReduction1}
    $\gamma_{\{Rk\}}(G) = kn$.
\end{claim}

\begin{subproof}
    Let~$f$ be any $\{k\}$-Roman dominating function of~$G$. 
    We first show that for every~${i \in [n]}$, ${f(V(H_i)) \geq k}$, which then implies that~$\gamma_{\{Rk\}}(G)\geq kn$. 
    To see this, notice that either each of~$v_i^1,\dots,v_i^k$ has strictly positive weight, or at least one of~$v_i^1,\dots,v_i^k$ has weight~$0$, in which case~$f(v_i^\top) + f(v_i^\bot) \geq k$.

    On the other hand, take the function~$f$ with ${f(v_i^\top) = \lceil k/2\rceil}$, ${f(v_i^\bot)=\lfloor k/2\rfloor}$ for each $i\in [n]$, and $0$ everywhere else. 
    Now each of~$v_i^1, \dots, v_i^k$ has weight~$k$ in its neighborhood, and each~$u_j$ has weight at least~$3 \cdot \lfloor k/2\rfloor \geq k$ in its neighborhood, so~$f$ is indeed a $\{k\}$-Roman dominating function. 
\end{subproof}

% Since for every graph~$G$, we have that $\gamma_{\{Rk\}}(G) \leq k\gamma(G)$, it follows that
% \begin{equation}
%     kn = \gamma_{\{Rk\}}(G)\le k\gamma(G) \,.
% \end{equation}
% In particular, we infer that $\gamma(G)\ge n$.

Moreover, the domination number of this graph satisfy the following.

\begin{claim}\label{c2:3satReduction1}
    $\gamma(G) = n$ if and only if the formula $\phi$ is satisfiable.
\end{claim}

\begin{subproof}
    Let~$D$ be a dominating set of size~$n$ in $G$. 
    For every~${i \in [n]}$, at least one vertex of~$H_i$ has to be in $D$, since otherwise $v_i^1$ would not be dominated. 
    Consequently, exactly one of~$v_i^\top, v_i^\bot$ is in~$D$ for each $i\in [n]$.
    Moreover, there are no other vertices in $D$.
    We define a satisfying assignment by setting~$x_i$ to be true if~${v_i^\top \in D}$ and false if~${v_i^\bot \in D}$. 
    Since every vertex $u_j$ has a neighbor in $D$, this indeed is an assignment for which all clauses~$c_1, \dots, c_m$ evaluate to true. 

    On the other hand, suppose that~$\phi$ is satisfiable by an assignment of truth values to the variables~$x_i$.
    We define~$D$ to be the set of all~$v_i^\top$ such that~$x_i$ is assigned true, and all~$v_i^\bot$ such that~$x_i$ is assigned false. 
    Now, $D$ is a set of size~$n$ which is dominating in $G$, since for each $i\in [n]$, each vertex in $V(H_i)\setminus D$ is adjacent to the vertex in $\{v_i^\top,v_i^\bot\}\cap D$, and for each $j\in [m]$, each vertex~$u_j$ is adjacent to at least one vertex in~$D$ by the definition of~$D$. 
\end{subproof}

By Claim \ref{cl:3satReduction1}, $\gamma_{\{Rk\}}(G) = kn$, and since by Claim \ref{c2:3satReduction1}, $kn= k\gamma(G) $ if and only if the formula is satisfiable, this completes the proof.
\end{proof}

\section{An alternative proof for ${k\ge 3}$ via perfect matchings in hypergraphs}
\label{sec:recognition}

A \emph{hypergraph} $H = (V, E)$ is a finite set of vertices $V$ together with a family $E$ of distinct, nonempty subsets of $V$, called edges. For $k \geq 2$, a \emph{$k$-uniform} hypergraph is one in which each edge has cardinality~$k$.
Hence, graphs are 2-uniform hypergraphs. 
A \emph{perfect matching} in a hypergraph $H= (V,E)$ is a set of pairwise disjoint hyperedges with union $V$.
A vertex $v$ in a hypergraph $H$ is said be \emph{isolated} if it belongs to no edge.
Given a hypergraph $H = (V,E)$ without isolated vertices, an \emph{edge cover} in $H$ is a set $E'\subseteq E$ of hyperedges such that every vertex of $H$ belongs to a hyperedge in $E'$  (see, e.g.,~\cite{Schrijver2003}).
The \emph{edge cover number} of $H$, denoted $\rho(H)$, is the minimum cardinality of an edge cover in $H$.
Restricting the above definitions to $2$-uniform hypergraphs, we obtain the usual definitions of perfect matchings, edge covers, and the edge cover number for graphs.

Given a graph $G$, the \emph{middle graph} of $G$, denoted $M(G)$, is the graph obtained from $G$ by subdividing every edge and joining two of the subdivided vertices by an edge if they share a neighbor (see~\cite{Hamada76}).
In this section we generalize several 
% some
known results on domination in middle graphs~\cite{sym12050751,MR4740902,MR4499617}
% middle graphs
to the hypergraph setting, connecting the existence of perfect matchings in hypergraphs with the property of being a $\{k\}$-Roman graph, for certain graphs related to hypergraphs.
This connection gives an alternative proof that the problem of recognizing $\{k\}$-Roman graphs is \textsf{NP}-hard for all $k\ge 3$, even when restricted to the class of split graphs.

\begin{sloppypar}
Kazemnejad et al.~\cite[Theorem 2.10]{MR4499617} proved that for every graph $G$ without isolated vertices, $\gamma(M(G)) = \rho(G)$.
Note that this implies that if $G$ is an $n$-vertex graph without isolated vertices, then 
$\lceil \frac{n}{2}\rceil\le \gamma(M(G))\le n-1$.
The former inequality was also proved by Basilio et al.~\cite{sym12050751}, who showed in addition that $\gamma(M(G)) = n/2$ if and only if $G$ has a perfect matching.
Kim~\cite{MR4740902} showed that if $G$ is an $n$-vertex graph, then \hbox{$\gr(M(G)) = n$}. 
These last two results imply the following result.
\end{sloppypar}

\begin{theorem}\label{middle-Italian}
Let $G$ be a graph with no isolated vertices.
Then, $M(G)$ is a $\{2\}$-Roman graph if and only if $G$ has a perfect matching.
\end{theorem}

We significantly generalize \zcref{middle-Italian} in \zcref{thm:k-Roman}.  
Let us also mention that together with known results on middle graphs~\cite{SkowronskySyslo1984} and matchings~\cite{Edomnds1965}, \zcref{middle-Italian} implies the following.

\begin{corollary}
The problem of recognizing $\{2\}$-Roman graphs is solvable in polynomial time when restricted to middle graphs.
\end{corollary}

\begin{proof}
Skowro\'nska and Sys\l{}o proved that middle graphs can be recognized in linear time (see~\cite{SkowronskySyslo1984}) and in the case of a positive answer, a graph $G$ such that the input graph equals $M(G)$ can be computed in linear time.
Since the existence of a perfect matching in a graph can be tested in polynomial time (see~\cite{Edomnds1965}), the result follows from \zcref{middle-Italian}.
\end{proof}

Our generalization of \zcref{middle-Italian} is based on the following definition, which unifies several concepts from the literature.

\begin{definition}\label{def-compatible-strongly-compatible}
Given a hypergraph $H=(V,E)$, a graph $G$ with vertex set $V\cup E$ is said to be \emph{compatible with $H$} if $V$ is an independent set in $G$ and, for every vertex $v\in V$ and hyperedge $e\in E$, the pair $\{v,e\}$ is an edge of $G$ if and only if $v\in e$. If, in addition, any two hyperedges $e_1,e_2\in E$ such that $e_1\cap e_2\neq \emptyset$ are adjacent in $G$, then $G$ is said to be \emph{strongly compatible with $H$}.
Note that there may be other edges between vertices in $E$.
\end{definition}

\zcref{def-compatible-strongly-compatible} captures the following concepts from the literature.
Suppose that $G$ is a graph compatible with a given hypergraph $H=(V,E)$.
Then:
\begin{itemize}
    \item If $G$ has the fewest possible number of edges, then $E$ is an independent set in $G$ and $G$ is precisely the \emph{bipartite incidence graph} of $H$\footnote{For a hypergraph $H=(V,E)$, the \emph{bipartite incidence graph of $H$} is the graph with vertex set $V\cup E$ and edges of the form $\{v,e\}$ with $v\in V$,  $e\in E$, and $v\in e$. 
    Clearly, this is a bipartite graph with bipartition $\{V,E\}$.} (see, e.g.,~\cite{KLMN2014}).
    \item If $G$ has the largest possible number of edges, then $G$ is the split graph obtained from the bipartite incidence graph of $H$ by turning $E$ into a clique.
    This construction was used by Bertossi in the proof that the dominating set problem is \textsf{NP}-complete in the class of split graphs (see~\cite{Bertossi1984}) as well as more recently by Boros et al.~(see~\cite{BGM2020}).
    \item If $H$ is $2$-uniform and $G$ is the unique minimal (with respect to the subgraph relation) graph that is strongly compatible with $H$, then $G$ is precisely the middle graph of~$H$.
\end{itemize}

By the last observation above, the following lemma generalizes the aforementioned result of Kim~\cite{MR4740902} (which corresponds to the case $k = 2$).

\begin{lemma}\label{lem:gamma_RkG}
Let $k\ge 1$, let $H=(V,E)$ be a $k$-uniform hypergraph, and let $G$ be a graph compatible with $H$.
Then $\gamma_{\{Rk\}}(G) = |V|$.
\end{lemma}

\begin{proof}
Since $H$ is $k$-uniform, assigning $1$ to each vertex in $V$ and $0$ to each vertex in $E$ yields a $\{k\}$-Roman dominating function in $G$.
Its total weight is $|V|$, hence, $\gamma_{\{Rk\}}(G) \le |V|$.

To prove the inequality $\gamma_{\{Rk\}}(G) \ge |V|$, consider an arbitrary $\{k\}$-Roman dominating function $f$ of $G$.
Let $V_0\subseteq V$ be the set of vertices $v\in V$ such that $f(v) = 0$ and let $F$ be the bipartite subgraph of $G$ with parts $V_0$ and $E$ such that for all $v\in V_0$ and $e\in E$, the vertices $v$ and $e$ are adjacent in $F$ if and only if they are adjacent in $G$.
Consider the following weight function on the edges of $F$: for each $v\in V_0$ and $e\in E$ such that $v\in e$, let $w(\{v,e\}) = f(e)$.
Let $W$ denote the sum of all the  edge weights of $F$.
Since $f$ is a $\{k\}$-Roman dominating function of $G$, for each vertex $v\in V_0$, the sum of the weights  of the edges of $F$ incident with $v$ is at least $k$.
Hence, $W\ge k|V_0|$.
On the other hand, since $H$ is $k$-uniform, for each $e\in E$, the value $f(e)$ appears on at most $k$ edges incident with $e$.
Therefore, $W\le k\cdot \left(\sum_{e\in E}f(e)\right)$. 
It follows that $\sum_{e\in E}f(e)\ge |V_0|$.
Since each vertex in $V\setminus V_0$ gets an $f$-weight of at least $1$, we obtain $\sum_{v\in V}f(v)\ge |V\setminus V_0|$.
Consequently, 
\[\sum_{x\in V(G)}f(x) = \sum_{v\in V}f(v)  + \sum_{e\in E}f(e) \ge  |V\setminus V_0| +|V_0| = |V|\,.\] 
Since $f$ was arbitrary, this shows that  $\gamma_{\{Rk\}}(G) \ge |V|$.
\end{proof}

Our next lemma generalizes the aforementioned result of Kazemnejad et al.~\cite[Theorem 2.10]{MR4499617} as well as the analogous result for the case when $G$ is a split graph (see the proof of~\cite[Theorem 1]{Bertossi1984}), corresponding to the minimum and maximum number of edges in $G$, respectively.

\begin{lemma}\label{lem:edge-cover}
Let $H=(V,E)$ be a hypergraph without isolated vertices and let $G$ be a graph strongly compatible with $H$.
Then $\gamma(G) = \rho(H)$.
\end{lemma}

\begin{proof}
First, we show that $\gamma(G)\le \rho(H)$.
Let $E'$ be a minimum edge cover in $H$.
It suffices to show that $E'$ is a dominating set in $G$.
This is easily verified:
First, every vertex $v\in V$ belongs to some hyperedge $e\in E'$, hence, $v$ is adjacent to $e$ in $G$.
Second, for every hyperedge $e\in E\setminus E'$, we can arbitrarily choose a vertex $v\in e$ and a hyperedge $e'\in E'$ such that $v\in e'$ to obtain an element $e'\in E'$ adjacent to $e$ in $G$ (adjacency holds since $v\in e\cap e'$ and $G$ is strongly compatible with $H$).
Hence, $E'$ is a dominating set in $G$, as claimed.

Next, we show that $\rho(H) \le \gamma(G)$.
Note that for every vertex $v\in V$, the neighborhood of $v$ in $G$ is a nonempty clique in $G$ (this is because $H$ is strongly compatible with $G$).
Hence, if $D$ is a dominating set in $G$ such that $v\in D\cap V$, then we can choose any neighbor $e$ of $v$ in $G$ and replace $v$ with $e$ in $D$ to obtain a dominating set $D'$ in $G$ such that $|D'|\le |D|$.
It follows that any minimum dominating set $D$ in $G$ that minimizes $|D\cap V|$ satisfies $D\subseteq E$. 
Let $D$ be such a minimum dominating set in $G$. 
Then, every vertex $v\in V$ has a neighbor in $D$.
Consequently, $D$ is an edge cover in $H$, implying $\rho(H) \le \gamma(G)$.
\end{proof}

The case $k = 2$ of our next lemma implies the aforementioned result of Basilio et al.~\cite{sym12050751} stating that if $G$ is an $n$-vertex graph without isolated vertices, then $\gamma(M(G)) = n/2$ if and only if $G$ has a perfect matching.

\begin{lemma}\label{lem:perfect-matching}
Let $k\ge 1$, let $H=(V,E)$ be a $k$-uniform hypergraph without isolated vertices, and let $G$ be a graph strongly compatible with $H$.
Then the following statements are equivalent.
\begin{enumerate}
    \item\label{item1-pm} $H$ has a perfect matching.
    \item\label{item3-eq} $\gamma(G) = |V|/k$.
    \item\label{item2-ub} $\gamma(G) \le |V|/k$.
\end{enumerate}
\end{lemma}

\begin{proof}
First, we show that \zcref{item1-pm} implies \zcref{item3-eq}.
Suppose first that $H$ has a perfect matching $M = \{e_1,\ldots, e_p\}$.
Note that $p = |V|/k$, since $H$ is $k$-uniform.

To prove that $\gamma(G) \le |V|/k$, we apply \zcref{lem:edge-cover} and prove instead that $\rho(H)\leq |V|/k$, which clearly follows since $M$ is an edge cover of $H$.
To see that $\gamma(G) \ge |V|/k$, note that for every vertex $v\in V$, the neighborhood of $v$ in $G$ is a nonempty clique in $G$ (this is because $H$ is strongly compatible with $G$).
Hence, if $D$ is a dominating set in $G$ such that $v\in D\cap V$, then we can choose any neighbor $e$ of $v$ in $G$ and replace $v$ with $e$ in $D$ to obtain a dominating set $D'$ in $G$ such that $|D'|\le |D|$.
It follows that any minimum dominating set $D$ in $G$ that minimizes $|D\cap V|$ satisfies $D\subseteq E$. Let $D$ be such a minimum dominating set in $G$. Since each vertex $e$ in $G$ has exactly $k$ neighbors in $V$, we infer that  $|D|\geq |V|/k$, and the equality follows. 

Trivially, \zcref{item3-eq} implies \zcref{item2-ub}.

Finally, we show that \zcref{item2-ub} implies \zcref{item1-pm}.
Suppose that $\gamma(G) \le |V|/k$. 
From \zcref{lem:edge-cover} it follows that $\rho(H)\leq |V|/k$. Let $E'$ be an edge cover of size $\rho(H)$. 
Suppose $E'$ is not a perfect matching.
Then $E'$ contains two edges with nonempty intersection. 
Since $H$ is $k$-uniform, a perfect matching has size $|V|/k$. 
Thus, $|E'|>|V|/k$, since otherwise, there would be vertices of $H$ not covered by $E'$, contradicting that $E'$ is an edge cover. 
Hence, $\rho(H) = |E'|>|V|/k$, a contradiction.
Therefore $E'$ is a perfect matching, as claimed. 
\end{proof}
\zcref{lem:gamma_RkG,lem:perfect-matching} have the following consequence, generalizing \zcref{middle-Italian}.

\begin{theorem}\label{thm:k-Roman}
Let $k\ge 1$, let $H=(V,E)$ be a $k$-uniform hypergraph without isolated vertices, and let $G$ be a graph strongly compatible with $H$.
Then $H$ has a perfect matching if and only if $G$ is a $\{k\}$-Roman graph.
\end{theorem}

For all $k\ge 2$, the problem of determining if a given $k$-uniform hypergraph has a perfect matching is equivalent to the decision problem \textsc{Exact Cover by $k$-Sets}, which is known to be \textsf{NP}-complete for all $k \geq  3$, even when restricted to $k$-partite hypergraphs~\cite{GAREY}.
Hence, \zcref{thm:k-Roman} yiels the following alternative proof of   \zcref{thm:k-Roman-NP-hard-via-3SAT} when $k\geq 3$.

\begin{theorem}\label{thm:k-Roman-NP-hard}
For every fixed $k\ge 3$, the problem of recognizing $\{k\}$-Roman graphs is \textsf{NP}-hard, even when restricted to the class of split graphs.
\end{theorem}

\begin{proof}
Let $H=(V,E)$ be an input to the \textsc{Exact Cover by $k$-Sets} problem, that is, a $k$-uniform $k$-partite hypergraph.
Let $G$ be the graph strongly compatible with $H$ such that $E$ forms a clique.
Since $V$ is an independent set and $E$ a clique in $G$, the graph $G$ is a split graph.
By \zcref{thm:k-Roman}, $H$ has a perfect matching if and only if $G$ is a $\{k\}$-Roman graph.
Hence, determining if $G$ is a $\{k\}$-Roman graph is \textsf{NP}-hard.
\end{proof}

\section{Characterizing \kr suns and their complements}
\label{sec:italian}

In this last section we characterize the \kr   property within two specific subclasses of split graphs: suns and their complements.
For an integer $t\ge 3$, the \emph{$t$-sun} is a split graph on $2t$ vertices with a split partition $(K,I)$ where $K = \{u_1, \ldots, u_t\}$ is a clique and $I = \{w_1,\ldots ,w_t\}$ is an independent set such that $u_i$ is adjacent to $w_j$ if and only if $i=j$ or $i=j+1$ (mod $t$).

    Note that the $t$-sun has a unique split partition.
    We denote the $t$-sun by~$S_{t}$ and its complement\footnote{The \emph{complement} of a graph $G = (V,E)$ is the graph $\overline{G}$ with vertex set $V$ in which two distinct vertices are adjacent if and only if they are not adjacent in $G$.} by $\overline{S_{t}}$.

    See \zcref{fig:t-suns} for examples of $t$-suns.

    \begin{figure}[htbp]
        \centering
        \includegraphics[width=0.5\textwidth]{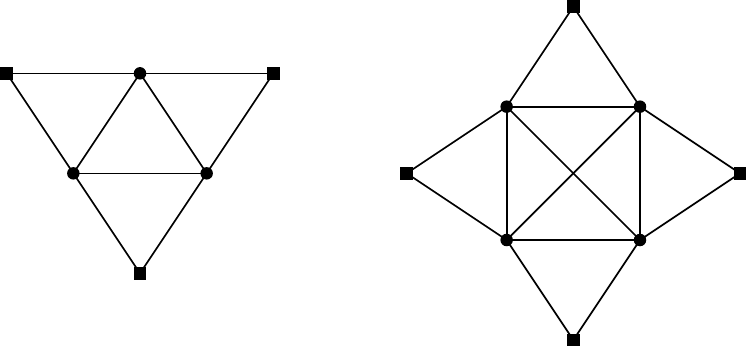}
        \caption{Examples of the $3$-sun (left) and the $4$-sun (right).}
        \label{fig:t-suns}
    \end{figure}

Some of the proofs will make use of the following general observation relating the classes of $\{k\}$-Roman graphs for different values of $k$.

\begin{observation}\label{obs:decreasing-k}
For $k\geq 2$, if $G$ is a $\{k+1\}$-Roman graph, then $G$ is a $\{k\}$-Roman graph.
\end{observation}

\begin{proof}
 Suppose that $G$ is not a $\{k\}$-Roman graph and let $f$ be a $\gamma_{\{R(k)\}}$-function of $G$.
 Then, $w(f)<k\gamma(G)$.
 Let $D$ be a minimum dominating set in $G$ and consider the function $f'\colon V(G)\to \{0,1,\ldots,k+1\}$ defined as \[f'(v)=\left\lbrace\begin{array}{ll}
     f(v)+1 & \text{ if } v\in D  \textnormal{ and}\\
     f(v) & \text{ otherwise}.\\
 \end{array}\right.\]
 Then, $f'$ is a $\{k+1\}$-Roman function of $G$ and $$w(f')=w(f)+|D|=w(f)+\gamma(G)<k\gamma(G)+\gamma(G)=(k+1)\gamma(G).$$ Therefore $G$ is not a $\{k+1\}$-Roman graph.
\end{proof}

\zcref{obs:decreasing-k} leads to the question whether there is a graph that is {$\{k\}$-Roman} for all $k\ge 2$.
However, this is not the case, since if a graph $G$ is $\{k\}$-Roman, then $k\le k \gamma(G)= \gamma_{\{Rk\}}(G)\le |V(G)|$, where the last inequality follows from the observation that assigning~$1$ to each vertex of $G$ yields a $\{k\}$-Roman dominating function of $G$.
On the other hand, it can be easily verified that the $n$-vertex complete graph is $\{k\}$-Roman for all $2\le k\le n$. 

\medskip
We next relate suns with cycles via one of the key notions of the previous section, that of strong compatibility (cf.~\zcref{def-compatible-strongly-compatible}).

\begin{observation}\label{obs:cycles-and-suns}
  The cycle graph $C_t$ is a $2$-uniform hypergraph such that $S_{t}$ is strongly compatible with $C_t$.   
\end{observation}

We now have everything ready to characterize the \kr property of suns: it turns out that for all $k\ge 2$ and $t\ge 3$, the $t$-sun $S_{t}$ is $\{k\}$-Roman if and only if $k = 2$ and $t$ is even.

\begin{theorem}
\label{lem:tsun}
For every integer $t\ge 3$, the following holds.
\begin{enumerate}
    \item The graph $S_{t}$ is a \rg{} if and only if $t$ is even. 
  
\item The graph $S_{t}$ is not a $\{k\}$-Roman graph for any $k\ge 3$.  
\end{enumerate}
\end{theorem}

\begin{proof}
 Since the cycle graph $C_t$ has a perfect matching if and only if $t$ is even, the first item follows directly from \zcref{obs:cycles-and-suns} and \zcref{thm:k-Roman}.
 Let now $k\ge 3$.
         Since by \zcref{obs:decreasing-k}, every $\{k\}$-Roman graph is a $\{k-1\}$-Roman graph, no $t$-sun with $t$ odd is a $\{k\}$-Roman graph, since it is not $\{2\}$-Roman, by the first item.
         Assume now that $t\ge 4$ is even, let $(K,I)$ be the split partition of $\sun{t}$, let $u$ be any vertex in $K$ and consider the function $f\colon V(\sun{t})\to \{0,1,2\}$ such that $f(v)=1$ if $v\in I\cup\{u\}$ and $f(v)=0$ otherwise. Because every vertex in $K$ has two neighbors in $I$ and either is $u$ or is adjacent to $u$, $f$ is a $\{3\}$-Roman dominating function with weight $f(\sun{t})=t+1$. Thus, $\gamma_{\{R3\}}(\sun{t})\le t+1$. 
         By \zcref{obs:cycles-and-suns} and \zcref{lem:edge-cover}, we obtain $\gamma(\sun{t})=\rho(C_t)=\lceil\frac{t}{2}\rceil$.
         It follows that $\gamma_{\{R3\}}(\sun{t})\neq 3\gamma(\sun{t})$, since $t+1<3\frac{t}{2} = 3\gamma(\sun{t})$.
         Therefore, since for even $t$, no $t$-sun is a $\{3\}$-Roman graph, no $t$-sun is a $\{k\}$-Roman graph for $k\ge 3$, by \zcref{obs:decreasing-k}.
     \end{proof}

Now let us turn our attention to complements of $t$-suns.

\begin{observation}\label{obs:cosun}
Let $t\ge 3$ and let $(K,I)$ be the split partition of~$S_{t}$.
Then, its complement $\overline{S_{t}}$ is a graph strongly compatible with a $(t-2)$-uniform hypergraph with vertex set~$K$.   
\end{observation}

The $(t-2)$-uniform hypergraph from \zcref{obs:cosun} has $t$ vertices and, hence, can only have a perfect matching if $t$ is a multiple of $t-2$.
Since this is not the case for any $t\ge 5$, \zcref{obs:cosun} and \zcref{thm:k-Roman} imply that for all $t\ge 5$, the graph $\overline{S_{t}}$ is not $\{t-2\}$-Roman (and, hence, also not $\{k\}$-Roman for any $k\ge t-2$, by \zcref{obs:decreasing-k}).
In the next theorem, we strengthen this observation by showing that, for integers $k\ge 2$ and $t\ge 3$, the graph $\overline{S_{t}}$ is $\{k\}$-Roman if and only if $k = 2$ and $t\in \{4,5\}$. 

    \begin{theorem}\label{th:comp-tsun}
    For every integer $t\ge 3$, the following holds.
\begin{enumerate}
    \item The graph $\overline{S_{t}}$ is a \rg{} if and only if $t\in\{4,5\}$.
\item The graph $\overline{S_{t}}$ is not a $\{k\}$-Roman graph for any $k\ge 3$.  
\end{enumerate}
\end{theorem}

The proof of \zcref{th:comp-tsun} relies on the following lemmas.

\begin{lemma}\label{claim:gamma_comp=2}
        If $t\ge 5$, then $\gamma({\cosun{t}})=2$. 
    \end{lemma}

    \begin{proof}
        Let~$(K,I)$ be the split partition of~$\sun{t}$. 
        Since $\sun{t}$ does not have isolated vertices, $\cosun{t}$ does not have universal vertices and so $\gamma(\cosun{t})>1$.         
        Let $v_1$, $v_2$ be two vertices in $I$ that have no neighbors in common in~$\sun{t}$.
        Consider the set of vertices $D=\{v_1,v_2\}$ of $\cosun{t}$. Note that, in $\cosun{t}$, vertex $v_1$ is adjacent to every vertex but its two neighbors in $\sun{t}$.
        Furthermore, since $v_2$ has no neighbors in common with $v_1$ in $\sun{t}$, vertex $v_2$ is adjacent in $\cosun{t}$ to the non-neighbors of~$v_1$. Therefore $D$ is a dominating set in $\cosun{t}$, implying that $\gamma(\cosun{t})=2$.
    \end{proof}

\begin{proof}[Proof of \zcref{th:comp-tsun}.]
       For $t=3$, it is easy to check that $\gr(\cosun{3}) \le 4 < 6 = 2 \gamma(\cosun{3})$.
       Hence, by \zcref{obs:decreasing-k}, $\cosun{3}$ is not $\{k\}$-Roman for any $k\ge 2$.

        For $t=4$, $\cosun{4}$ is isomorphic to $\sun{4}$, and it is clear from \zcref{lem:tsun} that $\cosun{4}$ is a \rg{} and not a $\{k\}$-Roman graph for any $k\ge 3$.

        Suppose that $t=5$. 
        To prove the first part of the statement, observe that by \zcref{claim:gamma_comp=2}, $\gamma(\cosun {5})= 2$, hence, showing that $\cosun {5}$ is a $\{2\}$-Roman graph is equivalent to showing that $\gr(\cosun{5})\ge 4$. 
        To this end, let $(K,I)$ be the split partition of $\cosun{5}$ and let $f$ be a $\gr$-function of $\cosun{5}$. 
        If $f(K)<2$, then $f(x)\geq 1$ for every $x\in I$, since $I$ is an independent set, and consequently $f(I)\geq 5$, which contradicts that~$5 \leq f(K \cup I) = \gr(\cosun{5}) \leq 2 \gamma(\cosun{5}) = 4$. 
        Hence, $f(K)\ge 2$.
        
        Suppose first that there exists $v\in K$ with $f(v) = 2$. Then, there exist exactly two vertices $x,y\in I$ that are not adjacent to $v$ in $\cosun{5}$. Since by the definition of a $\{2\}$-Roman dominating function, we need a weight of at least $2$ on $N[\{x,y\}]$, we conclude that in this case $\gr(\cosun{5})=f(\cosun{5})\ge 4$.
        
        Suppose now that there exists $v_1,v_2\in K$ with $f(v_1)=f(v_2)=1$, and note that $v_1$ and $v_2$ can have either one or two common neighbors in $I$.
        Consider first the case when they have two common neighbors in $I$. 
        Then there exists $x\in I$ that is adjacent to neither $v_1$ nor $v_2$ (see \zcref{fig:co-5-suns}). 

    \begin{figure}[htbp]
       \centering
        \includegraphics[scale=1]{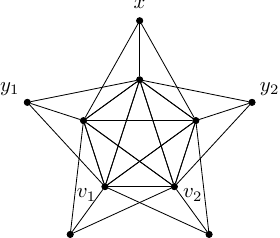}
       \caption{The complement of the $5$-sun with vertices $v_1,v_2\in K$ and $x,y_1,y_2\in I$, where vertex $x$ is not adjacent to $v_1$ nor $v_2$, vertex $y_1$ is adjacent to $v_1$ but not to $v_2$, and vertex $y_2$ is adjacent to $v_2$ but not to $v_1$.}
       \label{fig:co-5-suns}
   \end{figure}
   
        Now either~$f(N[x]) \geq 2$, in which case~$\gr(\cosun{5}) = f(\cosun{5})\ge 4$, or $f(x) = 1$. 
        In that case, since $v_1 \in N[y_1]$ and~$f(v_1) = 1$, we can observe that $f(N[y_1]) \geq 2$ (either $f(y_1) \geq 1$, or another neighbor of~$y_1$ has weight at least~$1$), where $y_1$ is adjacent to $v_1$ and not to $v_2$.
        So in this case~$\gr(\cosun{5}) = f(\cosun{5})\ge 4$ as well. 
    It remains to consider the case when $v_1$ and $v_2$ have one common neighbor in $I$. 
    Denote by $x_i$, $i\in\{1,2,3,4\}$, the remaining vertices in $I$. 
        Now each~$x_i$ either has weight at least~$1$, or there is a vertex distinct from~$v_1$ and~$v_2$ in its neighborhood that has weight at least~$1$. 
        Since the vertices $x_i$ have no common neighbor, we need a weight of at least $2$ on $(\bigcup_{i \in \{1,2,3,4\}} N[x_i]) \setminus \{v_1, v_2\}$, which will again result in $\gr(\cosun{5})=f(\cosun{5})\ge 4$.
        
        For the second part of the statement, by \zcref{obs:cosun} $\cosun{5}$ is strongly compatible with a 3-uniform hypergraph $H$. 
        Then, from \zcref{thm:k-Roman}, $\cosun{5}$ is not a $\{3\}$-Roman graph since the hypergraph $H$ does not have a perfect matching. 
        And again by \zcref{obs:decreasing-k}, $\cosun{5}$ is not a $\{k\}$-Roman graph for any~${k \geq 3}$.

        Finally, consider the case $t\ge 6$. 
        Now, let $(K,I)$ be the split partition of $\sun{t}$.
        Since $t\ge 6$, we can always choose vertices $x,y,z\in I$ such that no two of them have a common neighbor in $\sun{t}$. 
        Let $f\colon V(\cosun{t})\to \{0,1,2\}$ such that $f(x)=f(y)=f(z)=1$ and $f(v)=0$ otherwise. 
        Note that, if $v\in K$, then $v$ is adjacent in $\cosun{t}$ to at least two vertices among $x,y$, and $z$, thus $f(N(v)) \ge 2$. 
        Next, if $v\in I\setminus\{x,y,z\}$, then, since $I$ is a clique in $\cosun{t}$, vertex $v$ will be adjacent in $\cosun{t}$ to the vertices $x,y,z$, i.e., $f(N(v)) = 3$. 
        Therefore, $f$ is a $\{2\}$-Roman dominating function of $\cosun{t}$, implying $\gr(\cosun{t})\le 3$.
        From \zcref{claim:gamma_comp=2} it follows that $\cosun{t}$ is not a \rg{}, since $\gr(\cosun{t})\le 3 < 4= 2\gamma(\cosun{t})$.
        By \zcref{obs:decreasing-k}, $\cosun{t}$ is not a $\{k\}$-Roman graph for any~${k \geq 2}$.
    \end{proof}

\paragraph{Acknowledgements}

This work is supported in part by the Slovenian Research and Innovation Agency (I0-0035, research programs P1-0285, P1–0383, and P1-0404, and research projects J1-60012, J1-70035, J1-70046, and N1-0370), by the Agencia Nacional de Promoción de la Investigación, el Desarrollo Tecnológico y la Innovación through project PICT-2020-SERIE A-03032, and by Universidad Nacional de Rosario through project 80020220700042UR.

\bibliographystyle{abbrv} 
\bibliography{biblio} % Entries are in the "biblio.bib" file

\end{document}